\documentclass[11pt]{amsart}
\usepackage{amsthm,amsmath,amssymb,amscd,graphicx,epsfig}

{\end{list}}
{
   \newtheorem{theorem}{Theorem}[section]
   
   \newtheorem{lemma}[theorem]{Lemma}

   \newtheorem{corollary}[theorem]{Corollary}

}
{\theoremstyle{definition}

   \newtheorem{example}[theorem]{Example}
   \newtheorem{definition}[theorem]{Definition}

}

\newcommand{\ZZ}{{\mathbb{Z}}}

\newcommand{\bc}{{\mathbf{c}}}
\newcommand{\bd}{{\mathbf{d}}}
\newcommand{\ba}{{\mathbf{a}}}

\newcommand{\cO}{{\mathcal O}}

\newcommand{\ccd}{\widetilde{\psi}_{u,v}(\bc,\bd)}

\begin{document}
\title{On the complete $\bc\bd$-index of a Bruhat interval}

\author{Kalle Karu}
\thanks{The author was partially supported by the NSERC Discovery grant}
\address{Department of Mathematics\\ University of British Columbia \\
  1984 Mathematics Road\\
Vancouver, B.C. Canada V6T 1Z2}
\email{karu@math.ubc.ca}

\begin{abstract}
We study the non-negativity conjecture of  the complete $\bc\bd$-index
of a Bruhat interval defined by Billera and Brenti. For each
$\bc\bd$-monomial $M$ we construct a set of paths, such that if a
"flip condition" is satisfied, then the number of these paths is the
coefficient of the monomial $M$ in the complete $\bc\bd$-index. When
the monomial contains at most one $\bd$, then the condition follows
from Dyer's proof of Cellini's conjecture. Hence the coefficients of
these monomials are non-negative. We also relate the flip condition to
shelling of Bruhat intervals.
\end{abstract}

\maketitle

\section{Introduction}

Let $(W,S)$ be a Coxeter system and $u< v$ two elements in $W$
related in the Bruhat order. Billera and Brenti in \cite{BB} define a
polynomial $\ccd$ in the non-commuting variables $\bc, \bd$, called
the complete $\bc\bd$-index of the interval $[u,v]$. They conjecture
that this polynomial has non-negative coefficients. In this article we
study the non-negativity conjecture by constructing for each interval
$[u,v]$ in the Bruhat order and each $\bc\bd$-monomial $M$ a set of
paths $T_M(u,v)$, such that if a condition, called the {\em flip
  condition} is satisfied, then the number of paths in $T_M(u,v)$ is
equal to the coefficient of $M$ in the complete $\bc\bd$-index
$\ccd$. We conjecture the flip condition to be true for all intervals
and all monomials, which then would imply the non-negativity
conjecture.   

Using the notation explained in the next section, we briefly describe the flip
condition in its different forms and give evidence for it to hold. To
construct the set of paths $T_M(u,v)$ we need to fix a
reflection order $\cO$. Let $\overline{T}_M(u,v)$ be the set of paths
constructed using the reverse order $\overline{\cO}$. By induction on
the length of $[u,v]$, both sets have the same number of paths, equal
to the coefficient of the monomial $M$ in $\ccd$. Let 
$F: T_M(u,v) \to \overline{T}_M(u,v)$ be a bijection, called a
flip. The (strong) 
flip condition states that if $M$ starts with $\bc$, then one can
choose $F$ in such a way that if 
$F(x)=y$, then the first reflection in $x$ is less than or equal to the first
reflection in $y$. (This condition is then used to define
$T_{M'}(w,v)$ for longer intervals $[w,v]$, where $w<u<v$.) 

As a special case, consider $M=\bc^n$. Then $T_M(u,v)$ is the set of
ascending paths of length $n$ from u to $v$ (ascending with respect to
the reflection order $\cO$), and $\overline{T}_M(u,v)$ is the set of
descending paths of length $n$. A result of Dyer \cite{Dyer-C} states
that for any $x\in T_M(u,v)$ and $y\in \overline{T}_M(u,v)$, the first
reflection in $x$ always precedes the first reflection in $y$. Hence
the flip condition for this $M$ is true for any choice of $F$. As we
will see below, this result suffices to prove that $|T_M(u,v)|$ is 
the coefficient of $M$ in $\ccd$ in case the monomial $M$ contains at
most one $\bd$.  

The flip condition can be described in an equivalent form as
follows. The polynomial $\ccd$ is computed by summing the
ascent-descent sequences of all paths from $u$ to $v$. Let us fix a
reflection $t$ and sum the ascent-descent sequences of all those paths of
length $n$ from $u$ to $v$ that have their first reflection $\leq t$. This
sum can be expressed in the form  
\[ f_n(\bc,\bd)+ A g_{n-1}(\bc,\bd)\]
for some homogeneous $\bc\bd$-polynomials $f_n, g_{n-1}$ of degree
$n, n-1$, respectively. The (strong) flip condition is equivalent to
$g_{n-1}$ having non-negative coefficients. 

The second form of the flip condition can be related to the shelling
of the Bruhat interval. When $C$ is a regular $CW$-complex that is
topologically an $(n-1)$-ball or an $(n-1)$-sphere, then the
$\bc\bd$-index of of $C$ can be expressed in the form:
\[ f_n(\bc,\bd)+ \ba g_{n-1}(\bc,\bd),\]
for some homogeneous polynomials $f_{n}$ and $g_{n-1}$ with
non-negative coefficients \cite{Karu}. The Bruhat order on the
interval $[u,v]$ is 
shellable with respect to the lexicographic ordering of maximal chains
\cite{Dyer-S}. This implies that paths of 
maximal length from $u$ to $v$ with first reflection $\leq t$ are the
paths in the poset of a regular $CW$-complex $C$ that is topologically
a ball or a sphere. This means that the $f_n$ and $g_{n-1}$ in the
two formulas 
above coincide, and in particular that the flip condition holds for
paths of maximal length. 

We consider the two positive results described above as evidence for
the conjecture that the flip condition holds in general.

The approach to computing the $\bc\bd$-index by counting paths in
$T_M(u,v)$  is motivated by the theory of sheaves on posets
\cite{Karu}. One can 
define a sheaf on an appropriate poset constructed using length
$n$ paths from $u$ to $v$ in the Bruhat graph. Then the flip
condition states that one can carry out the same operations on this
sheaf as in the case of the constant sheaf on a Gorenstein* poset described in 
\cite{Karu}. The result of these operations is a vector space whose
dimension is the coefficient of $M$ in the $\bc\bd$-index. However,
since the sheaf for the Bruhat graph is constructed from the paths in
the graph, the operations reduce to counting paths with a given
ascent-descent sequence. Therefore, we only work with paths in the
Bruhat graph and do not mention sheaves again.

In the next section we recall the definition of the complete
$\bc\bd$-index in terms of a reflection order. We then construct the
sets $T_M(u,v)$ and give the condition for theses sets to count the
coefficient of $M$ in the complete $\bc\bd$-index.

\section{The complete $\bc\bd$-index.}

We fix a Coxeter system $(W,S)$ (see \cite{Humphreys, Bjorner-Brenti}) and a
reflection order $\cO$ (see \cite{Dyer-S}). The latter is a total
order on the set of reflections of $(W,S)$, satisfying a condition on
dihedral subgroups. The reverse of the order $\cO$ is also a
reflection order. We denote it by $\overline\cO$.

Let $l(x)$ be the length function on $W$. We write $u\prec v$ if
$l(u)<l(v)$ and $u^{-1}v$ is a reflection. The relation $\prec$
generates the Bruhat order on $W$. The Bruhat graph has vertex set $W$
and an edge from $u$ to $v$ if $u\prec v$.

Let $u < v$ in the Bruhat order. A path of length $n$ from $u$ to
$v$ in the Bruhat graph is a sequence
\[ x=(u=x_0\prec x_1\prec x_2\prec \cdots \prec x_n \prec x_{n+1}=v). \]
(Note a slightly unusual convention for the length. For example, the
path $(u\prec v)$ has length $0$.)  We let $B_n(u,v)$ be the set of
all paths of length $n$ from $u$ to $v$, and $B(u,v)= \cup_n
B_n(u,v)$. We label an edge $x_i\prec x_{i+1}$ with the reflection
$t_i = x_i^{-1}x_{i+1} $. The ascent-descent sequence of the path is
\[ w(x) = \beta_1 \beta_2 \ldots \beta_n,\]
where
\[ \beta_i = \begin{cases} A & \text{if $t_{i-1} < t_i$,} \\ D &
  \text{if $t_{i-1} > t_i$}. \end{cases}\] 
The reflections $t_i$ here are related by the reflection order $\cO$.
 
\begin{figure}[t] 
\centerline{\psfig{figure=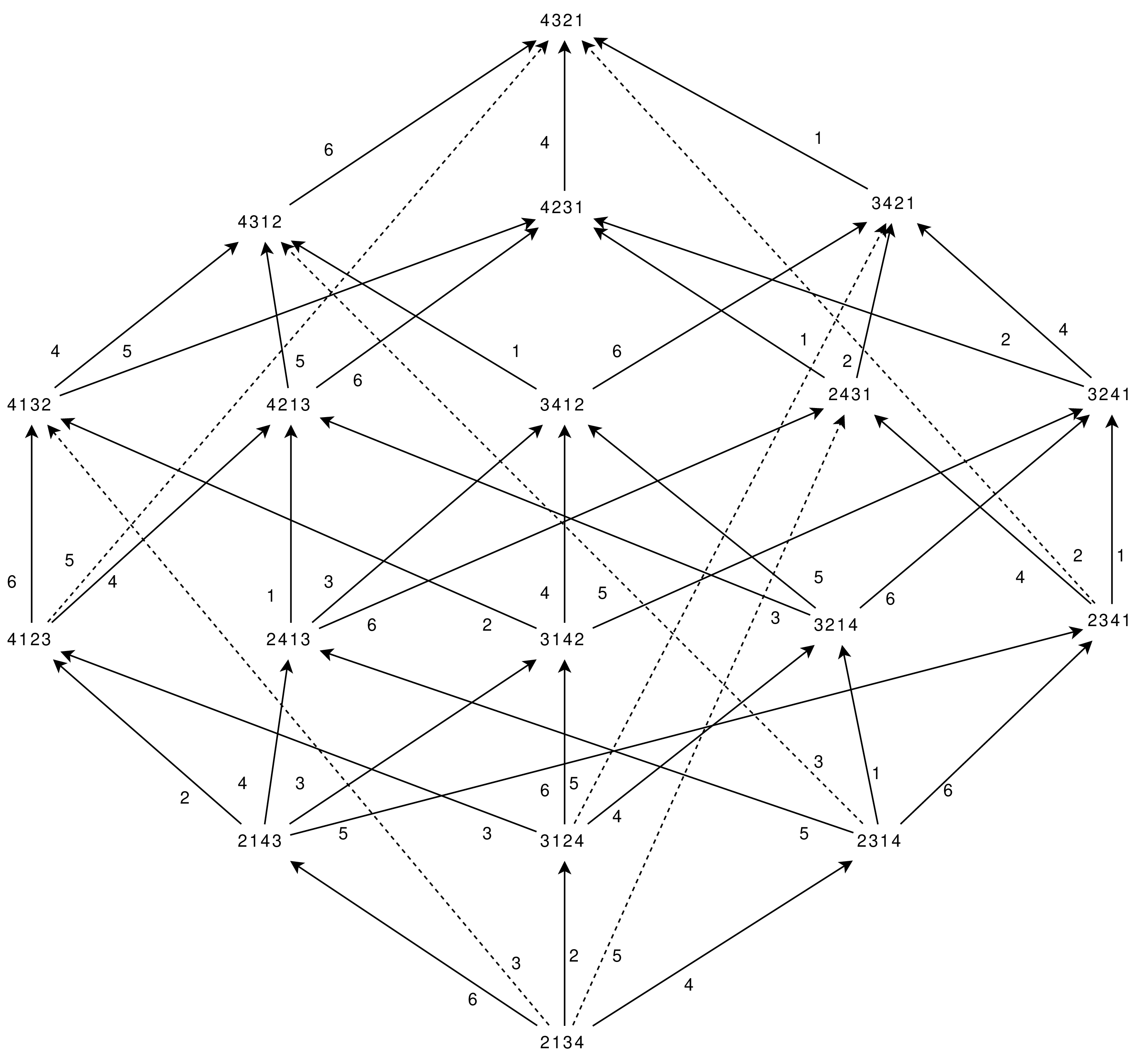,width=7cm}\label{fig-bruh}}
\caption{The Bruhat graph of the interval $[2134,4321]$ in $S_4$.}
\end{figure}

\begin{example} Figure~1 shows the Bruhat graph of the
  interval $[2134,4321]$ in the Coxeter system where the group is
  the symmetric group $S_4$ generated by transpositions
  $(12),(23),(34)$. The full Bruhat graph of this system can be found in
  \cite{Bjorner-Brenti}. The reflections here are the transpositions in $S_4$
  and they are ordered as follows: 
\[ (12)<(13)<(14)<(23)<(24)<(34).\]
We number the reflections so that $(12)$ has number $1$, $(13)$ has
number $2$, and so on. The edges in the Bruhat graph are then labeled
with the numbers of the corresponding reflections. For example, the path
$2134\prec 2143\prec 4123\prec 4132\prec 4312 \prec 4321$ has labels
$62646$, hence its ascent-descent sequence is $DADA$. The path
$2134\prec 3124 \prec 3421\prec 4321$ has labels $251$ and
ascent-descent sequence $AD$.
\end{example}

Let $\ZZ\langle A,D\rangle$ be the polynomial ring in non-commuting
variables $A$ and $D$. Summing the ascent-descent sequences of all
paths from $u$ to $v$ gives a polynomial in $A$ and $D$:
\[ \widetilde{\phi}_{u,v}(A,D) = \sum_{x\in B(u,v)} w(x).\]
The complete $\bc\bd$-index is obtained from this polynomial by a
change of variable. Let $\bc = A+D$ and $\bd = AD+DA$. This gives an
inclusion of rings
\[  \ZZ\langle \bc,\bd \rangle \subset \ZZ\langle A,D\rangle.\]
Billera and Brenti \cite{BB} prove that the polynomial
$\widetilde{\phi}_{u,v}(A,D)$ lies in this subring, hence can be
expressed in terms of $\bc$ and $\bd$:
\[ \widetilde{\phi}_{u,v}(A,D) = \ccd.\]
The polynomial $\ccd$ is the complete $\bc\bd$-index of the interval
$[u,v]$. It does not depend on the chosen reflection order $\cO$.

The rings $\ZZ\langle \bc,\bd \rangle$ and $\ZZ\langle A,D\rangle$ are
graded so that $A,D, \bc$ have degree $1$ and $\bd$ has degree $2$. We
define the involution $f\mapsto \overline{f}$ in the ring $\ZZ\langle
A,D\rangle$ by $\overline{f(A,D)} = f(D,A)$. Elements of $\ZZ\langle
\bc,\bd \rangle$ are invariant by this involution.

We will consider below homogeneous polynomials $p(A,D) \in \ZZ\langle
A,D\rangle$ that can be expressed in the form
$f_n(\bc,\bd)+g_{n-1}(\bc,\bd)D$ for some homogeneous
$\bc\bd$-polynomials $f_n$ and $g_{n-1}$. If such an expression
exists, then it is unique. We can recover $g_{n-1}$ by computing
$p(A,D)-\overline{p(A,D)} = g_{n-1}(\bc,\bd)(D-A)$, and then
subtracting 
$g_{n-1}(\bc,\bd)D$, we recover $f_n$. More generally, every
homogeneous $p(A,D) \in \ZZ\langle A,D\rangle$ of degree $n$ can be
expressed in a unique way as
\[ p(A,D) = f_n(\bc,\bd)+ f_{n-1}(\bc,\bd)D+f_{n-2}(\bc,\bd)
D^2+\cdots+ f_0 D^n\] 
for some homogeneous $\bc\bd$-polynomials $f_{i}$.

If $M(\bc,\bd)$ is a $\bc\bd$-monomial, consider the $AD$-monomial
$M(A,DA)$. This correspondence gives a bijection between
$\bc\bd$-monomials and $AD$-monomials in which every $D$ is followed
by an $A$. Below we will often use the letter $M$ to denote either the
$\bc\bd$-monomial $M(\bc,\bd)$ or the $AD$-monomial
$M(A,DA)$, with the distinction being clear for the context. For
example, we define $T_{M(\bc,\bd)}(u,v) = T_{M(A,DA)}(u,v)$.

\section{Coefficients of the complete $\bc\bd$-index}

Let $u<v$ in the Bruhat order and let $M(\bc,\bd)$ be a
$\bc\bd$-monomial of degree $n$. We wish to express the coefficient of
$M$ in $\ccd$ as a number of certain paths in $B_n(u,v)$. We start by
defining a number $s_M(x)$ for every path $x\in B_n(u,v)$, giving the
contribution of $x$ to the coefficient of $M$. The numbers $s_M(x)$
are in the set $\{-1,0,1\}$. We then study the case when $s_M(x)$ is
non-negative for every $x$ and call it the flip condition. If the flip
condition is satisfied, the number of paths $x$ with $s_M(x)=1$ is the
coefficient of $M$ in $\ccd$.

For any $AD$-monomial $w$, the number of paths $x\in B(u,v)$ with
$w(x) = w$ is equal to the number of paths $y\in B(u,v)$ with $w(y) =
\overline{w}$. Let $F=F_{u,v}: B(u,v)\to B(u,v)$ be an involution,
such that $w(F(x))=\overline{w(x)}$. We fix one such $F_{u,v}$ for
every $u<v$ and call it a flip.

Let $x=(u\prec x_1\prec x_2\prec\cdots\prec x_n\prec v) \in B_n(u,v)$,
and let $1\leq m\leq n$. We apply the flip $F_{x_m,v}$ to the tail of
$x$ to get $y=(u\prec x_1\prec \dots \prec x_m \prec y_{m+1}
\prec\cdots\prec y_n\prec v)$. If $w(x)=\beta_1 \cdots \beta_m \cdots
\beta_n$, then $w(y) =\beta_1 \cdots \beta_{m-1} \alpha_m
\overline\beta_{m+1} \cdots \overline\beta_n$, where $\alpha_m$ could
be either $A$ or $D$. Define:
\begin{align*} s_{m,A}(x) &= \begin{cases} 1  &\mbox{if } \beta_m = A \\ 
0 & \mbox{otherwise,} \end{cases} \\
 s_{m,D}(x) &= \begin{cases} 1  &\mbox{if } \beta_m = D, \alpha_m=A \\
 -1 &\mbox{if } \beta_m = A, \alpha_m=D \\ 0 &
 \mbox{otherwise.} \end{cases}
\end{align*} 

Now let $M(A,DA)$ be the
$AD$-monomial $\gamma_1\cdots\gamma_n$ and define
\[ s_M(x) = s_{M(A,DA)}(x)= \prod_{m=1}^n s_{m,\gamma_m}(x).\]

Let $x\in B_n(u,v)$, and let $y=F_{u,v}(x)$. Then the ascent-descent
sequence of $y$ when computed using the reverse reflection order
$\overline{\cO}$ is the same as the ascent-descent sequence of $x$
computed using the order $\cO$. Let us denote by
$\overline{s}_M(y)$ the number $s_M(y)$ computed as above, but using
the order $\overline{\cO}$. We say that $F$ is compatible with the
reflection order $\cO$ if $s_M(x) = \overline{s}_M(y)$ for any $u<v$, $M$
and $x$.

\begin{theorem} \label{thm-main} Assume that $F$ is compatible with
  the reflection order $\cO$. For any $\bc\bd$-monomial $M(\bc,\bd)$
  of degree $n$, the coefficient of $M$ in $\ccd$ is equal to  
\[ \sum_{x\in B_n(u,v)} s_M(x).\]
\end{theorem}

\begin{proof}
Write $N(A,D) = M(A,DA)$ and for $0\leq m\leq n$ let $N=N_m N_{n-m}$,
where $N_m, N_{n-m}$ are $AD$-monomials of degree $m, n-m$,
respectively. Define
\[ P_m = \sum_{x\in B_n(u,v)} w(u\prec x_1\prec\cdots\prec x_{m+1})
\cdot s_{N_{n-m}}(x_m\prec x_{m+1}\prec\cdots\prec x_n\prec v).\] 
Note that $P_n$ is the degree $n$ part of $\widetilde{\phi}_{u,v}$ and
hence can be expressed as a homogeneous $\bc\bd$-polynomial of degree
$n$. The statement of the theorem is that $P_0$ is the coefficient of
$M$ in $P_n$.

\begin{lemma} For $0\leq m\leq n$ there exist homogeneous
  $\bc\bd$-polynomials $f_m$ and $g_{m-1}$ of degree $m$ and $m-1$,
  respectively, such that  
\[ P_m = f_m(\bc,\bd)+g_{m-1}(\bc,\bd) D.\] 
Moreover, $P_{m-1}$ can be computed from $P_m$ as follows.
\begin{enumerate}
\item If $N_m$ ends with $A$, then $P_{m-1} =
  f_{m-1}(\bc,\bd)+g_{m-2}(\bc,\bd) D$, where $f_m =
  f_{m-1}\bc+g_{m-2}\bd$.
\item If $N_m$ ends with $D$, then $P_{m-1} = g_{m-1}(\bc,\bd)$.
\end{enumerate}
\end{lemma}

\begin{proof}
We use induction on $m$. When $m=n$, then $P_m$ is a homogeneous
$\bc\bd$ polynomial of degree $n$. Assume that $P_m =
f_m(\bc,\bd)+g_{m-1}(\bc,\bd) D$ and let us prove the "moreover"
statement.

If $N_m$ ends with $A$, let $x\in B_n(u,v)$ with $w(x)=\beta_1 \cdots
\beta_m \cdots \beta_n$. Then
\[ s_{AN_{n-m}}(x_{m-1}\prec x_m\prec\cdots\prec x_n\prec v)
= \begin{cases} s_{N_{n-m}} (x_m\prec\cdots\prec x_n\prec v)
  &\mbox{if } \beta_m = A \\  
0 & \mbox{otherwise.} \end{cases} \] Thus, to compute $P_{m-1}$ from
$P_m$, we consider only those monomials that end with $A$ and then
delete this last $A$. When contracting $P_m =
f_m(\bc,\bd)+g_{m-1}(\bc,\bd) D$ with $A$ from the right, we get
$f_{m-1}(\bc,\bd)+g_{m-2}(\bc,\bd) D$, where $f_m =
f_{m-1}\bc+g_{m-2}\bd$.

Now suppose $N_m$ ends with $D$. By induction on $m$, the
polynomials $f_m$ and $g_{m-1}$ depend only on $u<v$ and monomial
$M$, not on the reflection order or the flip $F$. Let us denote by
$\overline{w}$ and $\overline{s}_{N_{n-m}}$ the quantities computed using
the same flip $F$, but with the reverse reflection order
$\overline{\cO}$. This does not change the polynomial $P_m$. Then
$\overline{w}(x)$ is obtained from $w(x)$ by switching $A$ and $D$.
If
\[(x_m\prec y_{m+1}\prec \cdots\prec y_n\prec v) = F_{x_m,v}(x_m\prec
x_{m+1}\prec \cdots\prec x_n\prec v),\] 
then
\[\overline{s}_{N_{n-m}}(x_m\prec y_{m+1}\prec \cdots\prec y_n\prec v)
= {s}_{N_{n-m}}(x_m\prec x_{m+1}\prec \cdots\prec x_n\prec v)\] 
by the compatibility condition on $F$.

Let $F_m:B_n(u,v)\to B_n(u,v)$ be the involution that flips the tail
of a path:
\[ F_m( u\prec x_1\prec\cdots\prec x_{n}\prec v) = (u\prec
x_1\prec\cdots\prec  x_{m} \prec y_{m+1}\prec\cdots\prec y_n\prec
v).\] 
Since this is a bijection, we may compute $P_m$ with respect to
$\overline{\cO}$ by summing over $F_m(x)$.  We call this new
polynomial $Q_m$. By the previous discussion $Q_m=P_m$.
\begin{align*} 
Q_m &= \sum_{x\in B_n(u,v)} \overline{w}(u\prec x_1\prec\cdots\prec
y_{m+1}) \cdot \overline{s}_{N_{n-m}}(x_m\prec y_{m+1}\prec\cdots\prec
y_n\prec v) \\ &= \sum_{x\in B_n(u,v)} \overline{w}(u\prec
x_1\prec\cdots\prec y_{m+1}) \cdot s_{N_{n-m}}(x_m\prec
x_{m+1}\prec\cdots\prec x_n\prec v).
\end{align*}
Let us now compute $P_m-\overline{Q}_m = g_{m-1}(\bc,\bd)\cdot(D-A) $:
\[ \sum_{x\in B_n(u,v)} \big(w(u\prec x_1\prec\cdots\prec x_{m+1}) -
   {w}(u\prec x_1\prec\cdots\prec y_{m+1})  \big) \cdot
   {s}_{N_{n-m}}(x_m\prec \cdots\prec x_n\prec v).\] 
For $x\in B_n(u,v)$, let $w(u\prec x_1\prec\cdots\prec
x_{m+1})=\beta_1 \cdots \beta_m$ and $w(u\prec x_1\prec\cdots\prec
y_{m+1}) =\beta_1 \cdots \beta_{m-1} \alpha_m$. Then $x$ contributes
to this sum if and only if $\beta_m\neq \alpha_m$. The contribution is
 \[ \pm w(u\prec x_1\prec\cdots\prec x_{m}) (D-A) s_{N_{n-m}}(x_m\prec
 \cdots\prec x_n\prec v),\] 
 where the sign is positive if $\beta_m=D, \alpha_m=A$ and negative
 otherwise. Notice that, with the same sign,
\[ \pm {s}_{N_{n-m}}(x_m\prec \cdots\prec x_n\prec v) =
   {s}_{DN_{n-m}}(x_{m-1}\prec x_m\prec \cdots\prec x_n\prec v).\] 
This means that $P_{m-1} = g_{m-1}$.
\end{proof}

Now suppose $m$ is such that $M(\bc,\bd) = M_m(\bc,\bd)\cdot
M_{n-m}(\bc,\bd)$ where $M_m, M_{n-m}$ are $\bc\bd$-monomials of degree
$m,n-m$, respectively. Then the inductive computation of $P_{m-1}$
from $P_m$ in the lemma can be restated as follows. Let $f_m = f_{m-1}
\bc+g_{m-2}\bd$.
\begin{enumerate}
\item If $M_m$ ends with $\bc$, then
\[ P_{m-1} = f_{m-1}+g_{m-2}\cdot D.\]
\item If $M_m$ ends with $\bd$, then
\[ P_{m-2} = g_{m-2}.\]
\end{enumerate}
If we only consider the degree $m$ term $f_m$ of $P_m$, then in the
first case $f_{m-1}$ is obtained from $f_m$ by contracting with $\bc$
from the right. In the second case $f_{m-2}$ is obtained from $f_m$ by
contracting with $\bd$ from the right. It follows that $P_0$ is the
number that is obtained from $P_n$ by contracting with the monomial
$M$. In other words, $P_0$ is the coefficient of $M$ in $P_n$
\end{proof}

\section{Non-negativity of the complete $\bc\bd$ index.}

There are two problems with computing the coefficients of the complete
$\bc\bd$-index as described in the previous section. The first is that
the formula involves negative signs. The second problem is that it is not
clear how to define a flip $F$ that is compatible with the reflection
order.

In this section we define the "flip condition" requiring that
all terms $s_M(x)$ that go into the computation of the coefficient of
$M$ in the complete $\bc\bd$-index are non-negative. In this case we
define a set of paths $T_M(u,v)\subset B(u,v)$, such that $|T_M(u,v)|$
is the coefficient of $M$ in $\ccd$. It also turns out that the flip
condition gives an optimal
way of defining the flip $F_{u,v}$. The flip condition for the
interval $[u,v]$ only involves the flips $F_{w,v}$ where $u< w <v$,
hence this gives an inductive procedure 
for defining $F$, checking the flip condition and
constructing the set $T_M(u,v)$.

Let $u<v$ in the Bruhat order, and let $M(\bc,\bd)$ be a
$\bc\bd$-monomial of degree $n$. Let $M(A,DA)$ be the $AD$-monomial
$\gamma_1\cdots\gamma_n$.

\begin{definition} Let 
\[ T_M(u,v) = T_{\gamma_1\cdots \gamma_n}(u,v)= \{ x\in B_n(u,v) |
s_{m,\gamma_m}(x) = 1 \text{ for all } 1\leq m \leq n\}.\] 
\end{definition}

Using the definition of $s_{m,\gamma_m}$, a path $x$ lies in
$T_M(u,v)$ if and only if
\begin{enumerate} 
\item $w(x) = \gamma_1\cdots\gamma_n.$
\item For any $m$ such that $\gamma_m=D$, let
\[ (x_m\prec y_{m+1}\prec\cdots \prec y_n \prec v) =
F_{x_m,v}(x_m\prec\cdots\prec x_n\prec v).\] 
Then $w(x_{m-1}\prec x_m\prec y_{m+1}) = A$.
\end{enumerate}

The paths $x\in T_M(u,v)$ all satisfy $s_M(x)=1$. The following
condition implies that these are the only paths $x\in B_n(u,v)$
with $s_M(x)\neq 0$.

\begin{definition} The {\em flip condition} holds for the interval
  $[u,v]$ and monomial $M$ if for every $x\in B_n(u,v)$ the following
  is satisfied. If $s_{m,\gamma_m}(x) = -1$ for some $m$, then there
  exists $l>m$ such that $s_{l,\gamma_l}(x) = 0$. 
\end{definition}

This condition can be re-written using the definition of
$s_{m,\gamma_m}$ by saying that the flip condition is violated for
some $x\in B_n(u,v)$ if there exists $m$ such that
\begin{enumerate}
\item $(x_m\prec\cdots\prec x_n \prec v) \in
  T_{\gamma_{m+1}\cdots\gamma_n }(x_m,v)$. (Equivalently,
  $s_{l,\gamma_l}(x) = 1$ for $l>m$.)
\item $\gamma_m = D$ and if
\[ (x_m\prec y_{m+1}\prec\cdots \prec y_n \prec v) =
F_{x_m,v}(x_m\prec\cdots\prec x_n\prec v),\] 
then $w(x_{m-1}\prec x_m\prec x_{m+1}) = A$ and $w(x_{m-1}\prec
x_m\prec y_{m+1}) = D$. (Equivalently, $s_{m,\gamma_m}(x)=-1$.)
\end{enumerate}

From Theorem~\ref{thm-main} we now get:

\begin{corollary} Assume that $F$ is compatible with the reflection
  order $\cO$. If the flip condition holds for the interval $[u,v]$
  and monomial $M$, then $|T_M(u,v)|$ is the coefficient of $M$ in
  $\ccd$. \qed
\end{corollary}

\begin{example}
Consider the Bruhat interval $[u,v] = [2134,4321]$ shown in
Figure~1. We compute the sets $T_M(u,v)$ for different
monomials. When $M=\bc^n$, then $T_M(u,v)$ consists of all ascending
paths of length $n$ from $u$ to $v$. Thus,
$T_{\bc^2}(u,v)=\{346,235\}$ and $T_{\bc^4}(u,v)=\{23456\}$. When
$M=\bd$, we consider paths with ascent-descent sequence
$M(A,DA)=DA$. There are three such paths: $436,514,625$. We need to
check that when we flip such a path $(u\prec x_1\prec x_2 \prec v)$ to
$(u\prec x_1\prec y_2\prec v)$, then the resulting path must have
ascent-descent sequence $AD$. The flips of the three paths are the
paths $462,521,652$. Only the first one of these has the correct
ascent-descent sequence. Hence $T_\bd(u,v)=\{436\}$.

For a slightly longer computation, let us find the set
$T_{\bd^2}(u,v)$. For this we need to find all paths $(u\prec x_1\prec
x_2\prec x_3\prec x_4 \prec v)$ with ascent-descent sequence $DADA$
and check the ascent-descent sequences after applying the flips
$F_{x_3,v}$ and $F_{x_1,v}$.  There are $6$ paths with ascent-descent
sequence $DADA$,
\[ 62646, 64614, 63416, 63524, 41516,
41624.\] 
Applying the flip $F_{x_3,v}$ we get paths $62654, 64621, 63461,
63541, 41561, 41641.$ 
Among these, only the third and the fifth have the required ascent
descent sequence $DAAD$. This reduces the candidate paths to two:
$63416$ and $41516$. To check the flip $F_{x_1,v}$, we first need to
construct the sets $T_{ADA}(x_1,v)$ and
$\overline{T}_{ADA}(x_1,v)$. For the first path $63416$ we find that
$T_{ADA}(2143,v) = \{3416\}$ and $\overline{T}_{ADA}(2143,v) =
\{4361\}$. Thus, applying the flip to the path $63416$ gives
$64361$. This path does not have the required ascent-descent sequence
$ADAD$. For the second path $41516$, we find $T_{ADA}(2314,v) =
\{1516\}$ and $\overline{T}_{ADA}(2314,v) = \{5361\}$. The result of
applying the flip to $41516$ is $45361$ with the required
ascent-descent sequence $ADAD$. Thus, $T_{\bd^2}(u,v)=\{41516\}$. 

It follows from the discussion below and in the introduction that the
interval $[u,v]$ in this example satisfies the flip condition for any
monomial (either the paths have maximal length or the monomial
contains at most one $\bd$). This implies that the complete
$\bc\bd$-index $\ccd$ has non-negative coefficients. By the
computation above, 
$\ccd = 2\bc^2+\bd + \bc^4 + x\bc^2\bd+y\bc\bd\bc+z\bd\bc^2+\bd^2$ for
some $x,y,z\geq 0$. 
\end{example}

Let us now turn to the definition of the flip $F_{u,v}$. Note that in
the definition of $T_M(u,v)$ and the flip condition we only need to
apply the flip $F_{x_m,v}$ to paths $(x_m\prec\cdots\prec x_n\prec v)$
that lie in $T_{\gamma_{m+1}\cdots \gamma_n}(x_m, v)$. Since $F$ is
compatible with the reflection order, the result lies in
$\overline{T}_{\gamma_{m+1}\cdots \gamma_n}(x_m, v)$, where
$\overline{T}_M$ denotes the same set $T_M$ constructed using the
reverse reflection order $\overline{\cO}$. Assuming the flip
condition on $[u,v]$, these two sets have the same number of elements,
hence $F_{x_m,v}: T_{\gamma_{m+1}\cdots \gamma_n}(x_m, v)\to
\overline{T}_{\gamma_{m+1}\cdots \gamma_n}(x_m, v)$ is a
bijection. Thus, to define $F_{u,v}$, we only need a bijection
$T_M(u,v) \to \overline{T}_M(u,v)$ (and we can extend it to an
involution on $B_n(u,v)$ in an arbitrary way if we so wished).  Note
that such $F_{u,v}$ automatically satisfies the compatibility
condition because $s_M(x) = \overline{s}_M(F_{u,v}(x))=1$ for any
$x\in T_M(u,v)$ (and $s_M(x) = \overline{s}_M(F_{u,v}(x))=0$ for any
$x\notin T_M(u,v)$).

\begin{definition} Assume that the flip condition holds for the interval $[u,v]$
  and monomial $M$. Define the flip $F_{u,v}: T_M(u,v) \to
  \overline{T}_M(u,v)$ as the bijection that preserves the
  lexicographic ordering of paths. (The path with the smallest first
  reflection maps to a path with the smallest first reflection, using
  the order $\cO$ on both sides.) 
\end{definition}

The flip $F_{u,v}$ is optimal in the following sense. Consider a path
$(z\prec u \prec x_1\prec\cdots \prec x_n \prec v)$ and the
$AD$-monomial $\gamma_0\gamma_1\cdots \gamma_n$. We claim that if the
path $z$ violates the flip condition for $m=0$ and for the flip
$F_{u,v}$ defined above, then it violates the flip condition for any
$F$. Equivalently, if the flip condition holds for some $F$ then it
holds for the $F_{u,v}$ defined above. Indeed, the path violates the
flip condition for $m=0$ when $x=(u\prec x_1\prec \cdots \prec
x_n\prec v) \in T_M(u,v)$ and after applying $F_{u,v}$ we get the path
$(u \prec y_1\prec\cdots \prec y_n \prec v)$, such that $w(z\prec
u\prec x_1)=A$ and $w(z\prec u\prec y_1) = D$. This implies that
\[ u^{-1} y_1 < z^{-1} u < u^{-1} x_1.\]
Here $u^{-1} x_1$ is the first reflection in $(u\prec x_1\prec \cdots
\prec v) \in T_M(u,v)$ and $u^{-1} y_1$ is the first reflection in
$(u\prec y_1\prec\cdots\prec v)\in \overline{T}_M(u,v)$. Since
$F_{u,v}$ preserves ordering by first reflection, it follows that
$T_M(u,v)$ has fewer paths with first reflection less than $z^{-1} u$
than does $\overline{T}_M(u,v)$. Thus, no matter how the flip $F$ is
chosen, some $x\in T_M(u,v)$ with first reflection greater than $z^{-1}
u$ maps to $y\in \overline{T}_M(u,v)$ with first reflection smaller
than $z^{-1} u$. Hence the flip condition is violated for any
$F$. (Note that in the argument above $\gamma_0=D$, hence for
$\gamma_0\cdots\gamma_n$ to come from a $\bc\bd$-monomial, we need $\gamma_1=A$
and thus $M=\bc M'$. Then the path $(z\prec u \prec x_1\prec\cdots
\prec x_n \prec v)$ violates the flip condition for the interval
$[z,v]$ and monomial $\bd M'$.)

Let us say that the {\em strong flip condition} holds if for any $x\in
T_M(u,v)$ and $y=F_{u,v}(x)$, \[u^{-1} x_1 \leq u^{-1} y_1\]
whenever $M$ starts with $\bc$, $M=\bc M'$. This
condition is stronger than the flip condition because the flip
condition allows $u^{-1} y_1 < u^{-1} x_1$ as long as there is no
$z^{-1} u$ between them for some $z\prec u$.

When $M=\bc^n$, then the strong flip condition was proved by Dyer
\cite{Dyer-C}. Since to check the flip condition, we only need to
check the flip for each occurrence of $\bd$ in $M$, it follows from
this that the flip condition holds for any interval and any monomial
$M$ that contains at most one $\bd$. Thus, the coefficients of such
monomials are non-negative in any $\ccd$.

\section{Shelling of the Bruhat interval}
 
 In this section we give an equivalent formulation of the flip
 condition that is related to shelling of Bruhat intervals.

Let $t$ be a reflection. Denote
\[ B_n(u,v)_{\leq t} = \{(u\prec x_1\prec \cdots \prec x_n \prec v)
\in B_n(u,v) | u^{-1}x_1 \leq t\}.\] 
Also let
\[ T_M(u,v)_{\leq t} = T_M(u,v) \cap B_n(u,v)_{\leq t}, \qquad
\overline{T}_M(u,v)_{\leq t} = \overline{T}_M(u,v) \cap B_n(u,v)_{\leq
  t}.\]

\begin{theorem} The $AD$-polynomial 
\[ \widetilde{\phi}^{\leq t}_{u,v} =\sum_{x\in B_n(u,v)_{\leq t}} w(x)\]
can be expressed in the form $f_n(\bc,\bd)+A g_{n-1}(\bc,\bd)$
for some homogeneous $\bc\bd$-polynomials $f_n, g_{n-1}$ of degree $n,
n-1$, respectively. Assuming that the flip condition holds for the
interval $[u,v]$ and monomial $M$, then $|\overline{T}_M(u,v)_{\leq
  t}|$ is the coefficient of $M$ in $f_n$ and $|T_M(u,v)_{\leq t}|$ is
the coefficient of $M$ in $f_n+\bc g_{n-1}$.
\end{theorem}

Before we prove this theorem, let us derive an equivalent form of the
flip condition from it. Suppose the flip condition holds for the
interval $[u,v]$ and monomial $M=\bc M'$, but is violated for some
$(z\prec u \prec x_1 \prec \cdots\prec x_n \prec v)$ and monomial $\bd
M'$. Let $t=z^{-1} u$. Then, as in the previous section, we must have
\[ |T_M(u,v)_{\leq t}| < |\overline{T}_M(u,v)_{\leq t}|.\]
By the theorem, the difference between the two numbers is the
coefficient of $M$ in $\bc g_{n-1}$. Clearly this argument can also be
reversed to get an equivalent condition. For simplicity we will state
it without specifying the intervals and monomials.

\begin{corollary} The flip condition holds for all intervals and all
  monomials if and only if in the expression $\widetilde{\phi}^{\leq
    t}_{u,v} = f_n(\bc,\bd)+A g_{n-1}(\bc,\bd)$ 
the polynomial $g_{n-1}(\bc,\bd)$ has non-negative coefficients for
all intervals $[u,v]$ and all reflections $t=z^{-1}u$, where $z\prec
u$. \qed
\end{corollary}

The strong flip condition defined at the end of previous section is
equivalent to $g_{n-1}$ having non-negative coefficients for any
interval $[u,v]$ and any reflection $t$.
 
By the theorem, the flip condition also implies that the
polynomial $f_n(\bc,\bd)$ has non-negative coefficients.

\begin{proof}

To prove the first statement of the theorem, it suffice to show that
the $AD$-polynomial
\[ \sum_{u^{-1} x_1= t} w(x),\]
where the sum runs over all $x\in B_n(u,v)$ having $t$ as its first
reflection, has the stated form.  This sum can be written as
\begin{align*} & A\sum_{y\in B_{n-1}(x_1,v)_{\leq t}} w(y)+ D
  \sum_{y\in B_{n-1}(x_1,v)_{>t}} w(y) \\ 
 =& A\sum_{y\in B_{n-1}(x_1,v)_{\leq t}} w(y)+ D(
 \sum_{y\in B_{n-1}(x_1,v)} w(y) - \sum_{y\in B_{n-1}(x_1,v)_{\leq t}}
 w(y)) \\ =& (A-D) \sum_{y\in B_{n-1}(x_1,v)_{\leq t}} w(y)+
 D \sum_{y\in B_{n-1}(x_1,v)} w(y).
 \end{align*}
 Here the subscript $>t$ has similar meaning to $\leq t$. Using
 induction, we can write this as
 \begin{align*} & (A-D)(f_{n-1}(\bc,\bd) + Ag_{n-2}(\bc,\bd)) + D
   h_{n-1}(\bc,\bd) \\ 
 &= (2A-A-D) f_{n-1} + (A^2+AD-AD-DA) g_{n-2} +(A+D-A) h_{n-1} \\ &=
   (-\bc f_{n-1} -\bd g_{n-2}+\bc h_{n-1})+ A(2f_{n-1}+\bc g_{n-2} -
   h_{n-1}),
 \end{align*}
for some $\bc\bd$-polynomials $f_{n-1}, g_{n-2}, h_{n-1}$.

The proof of the second statement is very similar to the proof of
Theorem~\ref{thm-main}, so we only sketch it.

Let $N(A,D) = M(A,DA)$ and for $0\leq m\leq n$ write $N=N_m N_{n-m}$,
where $N_m, N_{n-m}$ are $AD$-monomials of degree $m,n-m$,
respectively. Define
\[ P_m^{\leq t}  = \sum_{x\in B_n(u,v)_{\leq t}} w(u\prec
x_1\prec\cdots\prec x_{m+1}) \cdot s_{N_{n-m}}(x_m\prec
x_{m+1}\prec\cdots\prec x_n\prec v).\] 
\[ Q_m^{\leq t}  = \sum_{x\in B_n(u,v)_{\leq t}} \overline{w}(u\prec
x_1\prec\cdots\prec x_{m+1}) \cdot \overline{s}_{N_{n-m}}(x_m\prec
x_{m+1}\prec\cdots\prec x_n\prec v).\]
Note that 
\[ P_n^{\leq t} = \widetilde{\phi}^{\leq t}_{u,v} = f_n(\bc,\bd)+A
g_{n-1}(\bc,\bd),\]
\[ Q_n^{\leq t} = \overline{P}_n^{\leq t} = f_n(\bc,\bd)+D
g_{n-1}(\bc,\bd).\]
On the other hand,
\[ P_0^{\leq t} = |T_M(u,v)_{\leq t}|, \quad Q_0^{\leq t} =
|\overline{T}_M(u,v)_{\leq t}|.\]

\begin{lemma} There exist $\bc\bd$-polynomials $f_m, g_{m-1}, h_{m-1},
  l_{m-2}$ of degree $m,m-1,m-1,m-2$, respectively, such that 
\[ P_m^{\leq t}  = (f_m+A g_{m-1})+(h_{m-1}+Al_{m-2})D.\]
For $m>0$, 
\[ Q_m^{\leq t}  = (f_m+D g_{m-1})+(h_{m-1}+Dl_{m-2})D.\]
Moreover, $P_{m-1}^{\leq t}, Q_{m-1}^{\leq t}$ can be computed from
$P_m^{\leq t}, Q_m^{\leq t} $ as follows.
\begin{enumerate}
\item If $N_m$ ends with $A$ and $m\geq 1$, then
\[ P_{m-1}^{\leq t}  = (f_{m-1}+A g_{m-2})+(h_{m-2}+Al_{m-3})D,\]
where $f_m = f_{m-1}\bc+h_{m-2}\bd$ and $g_{m-1} =
g_{m-2}\bc+l_{m-3}\bd$.
\item If $N_m$ ends with $A$ and $m=1$, let $P_1^{\leq t} = \alpha
  \bc +A\beta +\gamma D$ for $\alpha,\beta,\gamma \in \ZZ$. Then
  $P_0^{\leq t} = \alpha +\beta$ and $Q_0^{\leq t} = \alpha$
\item If $N_m$ ends with $D$, then
\[ P_{m-1}^{\leq t}  =  h_{m-1} +A l_{m-2}.\]\end{enumerate}
\end{lemma}

\begin{proof} 
If $N_m$ ends with $A$, we contract $P_m^{\leq t}$ and $Q_m^{\leq t}$
with $A$ from the right to get $P_{m-1}^{\leq t}$ and $Q_{m-1}^{\leq t}$.

If $N_m$ ends with $D$, then as before,
\[ P_{m-1}^{\leq t} (D-A) =  P_m^{\leq t} - \overline{Q}_m^{\leq t} =
h_{m-1}(D-A) +A l_{m-2}(D-A),\]
\[ Q_{m-1}^{\leq t} (D-A) =  Q_m^{\leq t} - \overline{P}_m^{\leq t} =
h_{m-1}(D-A) +D l_{m-2}(D-A).\]
\end{proof}

Let $m$ be such that $M = M_m M_{n-m}$, where $M_m, M_{n-m}$ are
$\bc\bd$-monomials of degree $m, n-m$, respectively. The lemma then
implies:
\begin{enumerate}
\item If $M_m$ ends with $\bc$ and $m>1$, then $f_{m-1}+A g_{m-2}$ is
  obtained by contracting $f_m+A g_{m-1}$ with $\bc$ from the right.
\item If $M_m$ ends with $\bc$ and $m=1$, then $P_0$ is obtained
  by contracting $f_{1}+\bc g_{0}$ with $\bc$ from the right and
  $Q_0$ is obtained by contracting $f_{1}$ with $\bc$ from the right.
\item If $M_m$ ends with $\bd$, then $f_{m-2}+A g_{m-3}$ is obtained
  by contracting $f_m+A g_{m-1}$ with $\bd$ from the right.
\end{enumerate}

It follows from this that if $P_n^{\leq t} = f_n+A g_{n-1}$, then
$P_0^{\leq t}$ is obtained from $f_n+\bc g_{n-1}$ by contracting with
$M$. Thus, $P_0^{\leq t} =|T_M(u,v)_{\leq t}|$ is the coefficient of
$M$ in $f_n+\bc g_{n-1}$. Similarly, $Q_0^{\leq t}$ is obtained from $f_n$ by
contracting with $M$, hence $Q_0^{\leq t}= |\overline{T}_M(u,v)_{\leq t}|$ is the
coefficient of $M$ in $f_n$.
\end{proof}

\bibliographystyle{plain}
\bibliography{complete}{}

\begin{thebibliography}{1}

\bibitem{BB}
L.~J. {Billera} and F.~{Brenti}.
\newblock {Quasisymmetric functions and Kazhdan-Lusztig polynomials}.
\newblock {\em ArXiv e-prints}, October 2007.

\bibitem{Bjorner-Brenti}
A.~Bj{\"o}rner and F.~Brenti.
\newblock {\em Combinatorics of {C}oxeter groups}, volume 231 of {\em Graduate
  Texts in Mathematics}.
\newblock Springer, New York, 2005.

\bibitem{Dyer-S}
M.~J. Dyer.
\newblock Hecke algebras and shellings of {B}ruhat intervals.
\newblock {\em Compositio Math.}, 89(1):91--115, 1993.

\bibitem{Dyer-C}
M.~J. Dyer.
\newblock {Proof of Cellini's conjecture on self-avoiding paths in Coxeter
  groups}.
\newblock {\em Preprint}, February 2011.

\bibitem{Humphreys}
J.~E. Humphreys.
\newblock {\em Reflection groups and {C}oxeter groups}, volume~29 of {\em
  Cambridge Studies in Advanced Mathematics}.
\newblock Cambridge University Press, Cambridge, 1990.

\bibitem{Karu}
K.~Karu.
\newblock The {$cd$}-index of fans and posets.
\newblock {\em Compos. Math.}, 142(3):701--718, 2006.

\end{thebibliography}

\end{document}